\documentclass[a4paper]{amsart}

\usepackage[utf8]{inputenc}
\usepackage{mathtools}
\usepackage{amssymb,amsfonts,amsmath}
\usepackage{enumerate}
\usepackage{graphicx}
\usepackage{verbatim}
\usepackage{tikz-cd}
\usepackage{comment}
\usepackage{tikz}
\usetikzlibrary{calc,positioning}
\usetikzlibrary{automata,positioning,arrows,calc,quotes}

\renewcommand\leq\leqslant
\renewcommand\geq\geqslant

\usepackage{enumitem}

\usepackage{yhmath} 

\usepackage{thmtools,thm-restate}
\usepackage[hidelinks]{hyperref}
\definecolor{darkblue}{rgb}{0.0, 0.0, 0.55}
\hypersetup{%
	colorlinks=true,
	linkcolor = darkblue,
	anchorcolor = darkblue,
	citecolor = darkblue,
	filecolor = darkblue,
	urlcolor = darkblue,
	linktoc=all
}

\usepackage[capitalise,noabbrev,sort]{cleveref}

\numberwithin{equation}{section}
\newtheorem{theorem}{Theorem}[section]
\newtheorem{corollary}[theorem]{Corollary}
\newtheorem{lemma}[theorem]{Lemma}

\newtheorem{proposition}[theorem]{Proposition}
\newtheorem{question}[theorem]{Question}

\crefalias{theoremx}{theorem}

\theoremstyle{definition}
\newtheorem{definition}[theorem]{Definition}
\newtheorem{notation}[theorem]{Notation}

\theoremstyle{remark}

\newcommand{\N}{\mathbb{N}}

\DeclareMathOperator{\Alt}{Alt}
\DeclareMathOperator{\Aut}{Aut}

\DeclareMathOperator{\EFRF}{EFRF}

\DeclareMathOperator{\id}{id}

\DeclareMathOperator{\pr}{pr}

\DeclareMathOperator{\RiSt}{RiSt}

\DeclareMathOperator{\sgn}{sgn}

\DeclareMathOperator{\St}{St}

\DeclareMathOperator{\Sym}{Sym}

\newcommand{\abs}[1]{| #1 |}
\newcommand{\absH}[1]{\Vert #1 \Vert_H}
\newcommand{\bigAbsH}[1]{\bigl\Vert #1 \bigr\Vert_H}
\newcommand\Set[2]{\{\,#1\mid#2\,\}}

\newcommand{\T}{\mathcal{T}}

\newcommand{\defeq}\coloneqq

\renewcommand{\epsilon}{\varepsilon}

\title[A branch group with unsolvable conjugacy problem]{A branch group with unsolvable conjugacy problem}
\author[A. Bishop]{Alex Bishop}
\email{alexbishop1234@gmail.com}
\author[E. Schesler]{Eduard Schesler}
\address{Karlsruhe Institute of Technology, Englerstr.\ 2, 76131 Karlsruhe, Germany}
\email{eduardschesler@googlemail.com}


\begin{document}
\begin{abstract}
We prove that every finitely generated residually finite group $G$ can be embedded in a finitely generated branch group $\Gamma$ such that two elements in $G$ are conjugate in $G$ if and only if they are conjugate in $\Gamma$.
As an application we construct a finitely generated branch group with solvable word problem and unsolvable conjugacy problem and thereby answer a question of Bartholdi, Grigorchuk, and \v{S}uni\'{k}.
\end{abstract}
\maketitle

\section{Introduction}

\noindent Since their formulation by Dehn~\cite{Dehn1911} in 1911, the word problem, the conjugacy problem, and the isomorphism problem for groups remained to be the three fundamental decision problems in group theory.
Still, the question of decidability of these problems within a class of groups $\mathcal{C}$ is often among the first questions that arises in the study of $\mathcal{C}$.
Among the classes of groups whose algorithmic problems received particularly great interest are those that themselves are defined via algorithmic concepts.
One instance of this is the class of (finite state, synchronous) automata groups, see e.g.~\cite{BartholdiGrigorchukNekrashevych2003,BartholdiGrigorchukSunic03,GrigorchukNekrashevichSushchanskii00} for some background on these groups.
Although the word problem is well-known to be solvable for finitely generated automata groups, the picture regarding the solvability of the conjugacy problem in automata groups is far from being complete.
First steps in this direction were achieved by Wilson and Zalesskii~\cite{WilsonZalesskii1997} who showed that for odd primes $p$ the conjugacy problem is solvable in certain $p$-groups that were introduced by Grigorchuk~\cite{Grigorchuk1986} and Gupta and Sidki~\cite{GuptaSidki83}.
This result was complemented by Rozhkov~\cite{Rozhkov1998} and Leonov~\cite{Leonov1998} who showed that the conjugacy problem is solvable for Grigorchuk's $2$-group, respectively for all members of
the family of Grigorchuk groups $G_{\omega}$ from~\cite{Grigorchuk1985} whose word problem is solvable.
In~\cite{GrigorchukWilson2000} Grigorchuk and Wilson unified and extended the latter results to more general classes of branch groups, including most of the classical examples of automata groups.

Another milestone in the understanding of the solvability of the conjugacy problem in automata groups and branch groups was provided by I.\ Bondarenko, N.\ V.\ Bondarenko, Sidki, and Zapata~\cite{BondarenkoBondarenkoSidkiZapata2013}, who showed that the conjugacy problem is solvable in every finitely generated group of bounded automata.
On the other hand, \v Suni\'c and Ventura~\cite{SunicVentura2012} gave concrete examples of exponential activity automata groups with unsolvable conjugacy problem.
As the above results suggest, many of the most-studied branch groups arise as (bounded) automata groups.
However, unlike the class of automata groups, the class of finitely generated branch groups contains uncountably many isomorphism types so that there exist finitely generated branch groups with unsolvable word problem.
It is therefore a natural question, raised by Bartholdi, Grigorchuk, and \v{S}uni\'{k}~\cite[Question 5]{BGS-branch}, whether the word problem is the only obstruction for the solvability of the conjugacy problem in the class of finitely generated branch groups.
The following result provides a negative answer to this question.

\begin{restatable}{theoremx}{MainTheorem}\label{thm:main-existence-branch}
There exist finitely generated branch groups with solvable word problem and unsolvable conjugacy problem.
\end{restatable}

As usual with results on the unsolvability of algorithmic problems, we will deduce Theorem~\ref{thm:main-existence-branch} from the unsolvability of a related problem.
In our case, this related problem is given by the conjugacy problem for finitely generated residually finite groups.
By a result of Miller~\cite[Theorem~9 on p.~31]{Miller1971} the latter problem is unsolvable even if we assume that the given residually finite group has solvable word problem.
In fact Miller constructs a finitely presented residually finite group $H$ with unsolvable conjugacy problem.
Recall that the relevance of finite presentability for the word problem of a residually finite group $G$ is that it enables us to decide whether a given assignment $f \colon S \rightarrow Q$ from a finite generating set $S$ of $G$ to a finite group $Q$ extends to a homomorphism $\widetilde{f} \colon G \rightarrow Q$.
This allows us to detect the non-triviality of an element $g \in G$ by enumerating all homomorphisms from $G$ to every symmetric group $\Sym(n)$ and applying them to $g$.
The latter procedure is a particular case of an algorithm that takes a word $w$ over the generators of $G$ as an input and, if $w \neq_G e$, produces a morphism $\varphi \colon G \rightarrow Q$ to a finite group $Q$ that satisfies $\varphi(w) \neq_Q e$.
Following~\cite{Rauzy2021}, we will say that a group $G$ is \emph{effectively residually finite} ($\EFRF$) if it admits such an algorithm.
For our purposes it will be useful to moreover assume that the algorithm does not terminate whenever the word $w$ represents the identity in $G$.
In this case, the group $G$ is referred to as an \emph{effectively residually finite group with co-semi-decidable word problem} ($\EFRF^{+}$), see Definition~\ref{def:eff-res-fin} for a more precise definition and~\cite{Rauzy21Computable,Rauzy2021} for more background on these properties.
If a finitely generated recursively presented group $G$ is $\EFRF^{+}$ it can be easily deduced that $G$ has a solvable word problem.

In order to take advantage of Miller's group $H$, we construct a finitely generated branch group $\Gamma$ and an embedding $\iota \colon H \rightarrow \Gamma$ for which two elements $g,h \in H$ are conjugate in $H$ if and only if $\iota(g)$ and $\iota(h)$ are conjugate in $\Gamma$.
Such an embedding is known as a Frattini embedding and guarantees that the unsolvability of the conjugacy problem in $H$ implies the unsolvability of the conjugacy problem in $\Gamma$.
Regarding this, we obtain Theorem~\ref{thm:main-existence-branch} as a consequence of the following result.

\begin{restatable}{theoremx}{EmbeddingTheorem}\label{thm:main-embedding-intro}
For every finitely generated residually finite group $G$, there exists a finitely generated branch group $\Gamma$ and a Frattini embedding $\iota \colon G \rightarrow \Gamma$.
Moreover, if $G$ is recursively presented and $\EFRF^{+}$, then $\Gamma$ can be chosen to be recursively presented and $\EFRF^{+}$ as well.
\end{restatable}

\noindent As pointed out by Rauzy~\cite[Theorem 2]{Rauzy2021}, there exist finitely generated residually finite groups with solvable word problem that do not embed in finitely generated effectively residually finite groups.
In particular, this tells us that the assumption on $G$ in Theorem~\ref{thm:main-embedding-intro} to be $\EFRF^{+}$ cannot be relaxed to the assumption of having a solvable word problem while keeping the effective residual finiteness for $\Gamma$.
However, to the best of our knowledge, it could still be true that the following question has an affirmative answer.

\begin{question}\label{quest:BH-for-branch}
Does every finitely generated residually finite group $G$ with solvable word problem embed in a finitely generated branch group with solvable word problem?
\end{question}

\noindent As a byproduct of Theorem~\ref{thm:main-embedding-intro}, we note the following direct consequence for the Boone--Higman conjecture, which asserts that a finitely generated group $G$ has solvable word problem if and only if $G$ embeds into a finitely presented simple group.

\begin{corollary}\label{cor:boone-higman}
If the Boone--Higman conjecture is true for the class of branch groups with $\EFRF^{+}$, then it is true for every group with $\EFRF^{+}$.
\end{corollary}

\noindent In view of Corollary~\ref{cor:boone-higman} it seems interesting and intriguing that it recently became a successful strategy to apply variations of typical properties of branch groups, such as self-similarity, in order to construct finitely presented simple groups that contain groups with those properties, see e.g.~\cite{BelkBleakMatucciZaremsky2023,BelkMatucci2024,Zaremsky2025}.

\subsection*{Notation}
In this paper, we write $\N = \{1,2,3,\ldots\}$ for the natural number starting at 1, and we write $\N_0=\{0,1,2,3,\ldots\}$ for the natural numbers including 0.
Suppose $g,h\in G$ are two elements of a group $G$, then we define the conjugate $g^h \coloneqq h^{-1}gh$.

\subsection*{Acknowledgements}
The authors are grateful to Claudio Llosa Isenrich, Roman Sauer, and Emmanuel Rauzy for a number of helpful discussions.
The first author acknowledges support from Swiss NSF grant 200020-200400,
and support from the University of Technology Sydney by an Honorary Appointment as Visiting Fellow.

\section{Effectively residually finite groups}

Let $G$ be a finitely generated, recursively presented, residually finite group and let $S$ be a finite symmetric generating set of $G$.
Given a word $w \in S^*$, we write $\overline{w} \in G$ to denote the element of $G$ that is represented by $w$.
For each word $w \in S^*$ with $\overline{w} \neq 1$, let $f_w \colon G \to Q_w$ be a homomorphism onto a finite group $Q_w$ that satisfies $f_w(\overline w) \neq 1$.
The set of words of length at most $n \in \N$ over $S$ that represent a non-trivial element in $G$ will be denoted by
\begin{equation}\label{eq:word-set-W_n}
    W_n \coloneqq \{ w\in S^* \mid \overline w \neq 1,\ |w|_S\leq n \}.
\end{equation}
For each $n\in \N$, we define the homomorphism
\[
    f_n' \colon G \rightarrow \prod \limits_{w \in W_n} Q_w,\ \ 
    g \mapsto (f_w(g))_{w \in W_n},
\]
whose image we denote by $Q_n$.
By corestricting $f_n'$, we obtain the surjective homomorphism
\begin{equation}\label{eq:varphi_n}
f_n \colon G \to Q_n.
\end{equation}
Since $G$ is residually finite, the sequence $(N_n)_{n \in \N}$ that consists of the kernels $N_n \defeq \ker(f_n)$ is a residual chain in $G$, i.e.\ it satisfies $\bigcap \limits_{n \in \N} N_n = 1$.

It is a direct consequence of the definition of $f_n$ in (\ref{eq:varphi_n}) that the word length of every non-trivial element in $N_n$ is bounded from below by $n+1$.
In what follows we work with the assumption that there is a Turing machine that takes $n \in \N$ as an input and returns the map $f_n$.
The existence of such a Turing machine can be guaranteed (see~\cref{lem:computable-homos}) by combining our assumption that $G$ is recursively presented with the following definition, which was introduced by Rauzy~\cite{Rauzy21Computable,Rauzy2021}.

\begin{definition}\label{def:eff-res-fin}
Let $G$ be the group with finite generating set $S$ as chosen at the beginning this section, then $G$ is \emph{effectively residually finite group with co-semi-decidable word problem} ($\EFRF^{+}$) if there exists a Turing machine that takes as input a word $w\in S^*$, and proceeds as follows:
\begin{enumerate}
    \item[$(\mathrm{E1})$] if $\overline w \neq 1$, then the machine outputs a homomorphism $f_w\colon G\to Q_w$, with the properties described in the previous paragraph;
    \item[$(\mathrm{E2})$] otherwise, $\overline w = 1$ and the Turing machine does not halt.
\end{enumerate}
If we relax these conditions by not specifying the output of the Turing machine for $\overline w = 1$, then $G$ is said to be an \emph{effectively residually finite group} ($\EFRF$).
\end{definition}

\subsection{Output format}\label{sec:output-formats}
In \cref{def:eff-res-fin}, it makes sense to speak of a Turing machine which outputs $f_w\colon G\to Q_w$ as such a homomorphism can be completely described by a finite amount of data.
In particular, we need only output the values $f_w(s) \in Q_w$ for each generator $s\in S$, where each such element of $Q$ can be described as an element of $\Sym(|Q_w|)$, for some fixed enumeration of $Q_w$.

\subsection{Computable sequences of normal subgroups}
From the definition of $\EFRF^{+}$ in \cref{def:eff-res-fin}, we may prove the following result which is used in the proof of our main theorems.

\begin{lemma}\label{lem:computable-homos}
Let $G$ be a finitely generated group with finite symmetric generating set $S$.
If $G$ is $\EFRF^{+}$ and recursively presentable, then there exists a Turing machine that takes as input a number $n \in \N$, and returns as output a homomorphism $f_n \colon G \to Q_n$ as described in (\ref{eq:varphi_n}).
\end{lemma}
\begin{proof}
Since $G$ is recursively presented, there exists a Turing machine that takes as input a word $w\in S^*$ and terminates if and only if $\overline w = 1$.
By running this machine in parallel with the machine in Definition~\ref{def:eff-res-fin}, we obtain a Turing machine $M$ that takes as input a word $w\in S^*$ and proceeds as follows:
\begin{itemize}
    \item if $\overline w \neq 1$, then the procedure returns a homomorphism $f_w \colon G \to Q_w$ onto a finite group $Q_w$ with $f_w(\overline{w}) \neq 1$,
    \item if $\overline w = 1$, then the procedure terminates and outputs `\verb!identity!'
\end{itemize}
Notice here that we intend that the output `\verb!identity!' cannot be confused with the specification of a homomorphism.

Equipped with such a Turing machine $M$, we can build another one that takes as input a number $n \in \N$, enumerates the words of length at most $n$ over $S$, applies $M$ to each of them, and returns the corestriction $f_n$ of the homomorphism
\[
    f_n' \colon G \rightarrow \prod \limits_{w \in W_n} Q_w,\ \ g \mapsto (f_w(g))_{w \in W_n}
\]
onto its image.
By construction, this newly constructed Turing machine fulfills the requirements of the Turing machine from the lemma.
\end{proof}

\begin{corollary}\label{cor:computable-wp}
Let $G$ be a finitely generated group with finite symmetric generating set $S$.
If $G$ is $\EFRF^{+}$ and recursively presentable, then the word problem is solvable for $G$.
\end{corollary}

\begin{proof}
Suppose we are given a word $w=s_1s_2\cdots s_n\in S^*$.
We may then use the Turing machine in \cref{lem:computable-homos} to find the homomorphism $f_n\colon G\to Q_n$ as in (\ref{eq:varphi_n}).
We then see that $f_n(\overline{w}) = f_n(\overline{s_1})f_n(\overline{s_2})\cdots f_n(\overline{s_n}) = 1$ if and only if $\overline{w}=1$ in $G$.
Thus, using $f_n\colon G \to Q_n$, we may decide the word problem for $w$.
\end{proof}

\section{Groups acting on rooted trees}

\noindent Our objective in this section is to develop notation and basic results concerning groups that act on rooted trees, which we use in \cref{sec:constructing-branch-groups} to construct the group $\Gamma$ in \cref{thm:main-embedding-intro}.
For more background on these groups, and in particular for background on branch groups, we refer the reader to~\cite{BGS-branch}.

\subsection{Rooted trees}\label{subsec:rooted-trees}

Let $X = (X_n)_{n \in \N}$ be a sequence of finite sets, each with cardinality $\abs{X_n} \geq 2$, which we will think of as alphabets.
By taking finite products over those alphabets we obtain the sets of words $X^{\ell} \defeq X_1 \times X_2 \times \cdots \times X_{\ell}$ of length $\ell$ for each $\ell \in \N_0$, where $X^{0} = \{\emptyset\}$.
The \emph{spherically homogeneous rooted tree associated to $X$}, denoted by $\T_X$, is the rooted tree with vertex set $X^{\ast} = \bigcup
_{\ell=0}^{\infty} X^{\ell}$ and root vertex $\emptyset$, where two vertices $v,w$ are connected by an edge if and only if there is a letter $x \in X_n$ for some $n \in \N$ such that either $v = wx$ or $w = vx$.
Let $\Aut(\T_X)$ denote the group of all automorphisms of $\T_X$ that fix the root. 
Note that the latter condition implies that the distance of a vertex $v$ to the root, which we call the \emph{level} of $v$, is preserved under the action of $\Aut(\T_X)$ and coincides with the word length $\abs{v}$ of $v$.
Thus for every subgroup $G \leq \Aut(\T_X)$ and every $\ell \in \N_0$, we have a natural homomorphism $\pi_{\ell} \colon G \rightarrow \Sym(X^{\ell})$.
On the other hand, every permutation $\sigma \in \Sym(X_{1})$ gives rise to an automorphism $\widehat{\sigma}$ of $\T_X$ by setting $\widehat{\sigma}(xw) = \sigma(x)w$ for every $x \in X_1$ and every $w \in X_2 \times X_3 \times \cdots \times X_{k}$ with $k \geq 2$.
Automorphisms obtained in this way are called \emph{rooted}.
More generally, we call a tree autormorphism $\alpha \in \Aut(T_X)$ \emph{finitary} if there is some $n \in \N_0$ and a permutation $\sigma \in \Sym(X^n)$, such that $\alpha(vw) = \sigma(v)w$ for every $v \in X^n$ and every $w \in X_{n+1} \times \ldots \times X_m$ with $m > n$.
In this case we will write $\alpha = \widehat{\sigma}$.
A natural counterpart to finitary automophisms are those automorphisms of $\T_X$ that are supported on a subtree of $\T_X$ whose vertex set consists of those words in $X^{\ast}$ that start with a given word $v \in X^{\ast}$.
Note that 
this
subtree is canonically isomorphic to the 
rooted tree $\T_X^{[\ell]}$ that is associated to the subsequence $(X_{\ell+1}, X_{\ell+2}, \ldots)$ of $X$, where $\ell = |v|$.
Using the tree $\T_X^{[\ell]}$, we can now introduce the following notation for the subtree of $\T_X$ that we have just described.
\begin{notation}\label{not:vTXl}
For each $\ell \in \N_0$ and each $v \in X^{\ell}$ we write $v\T_X^{[\ell]}$ to denote the subtree of $\T_X$ whose vertex set consists of those words in $X^{\ast}$ that start with $v$.
\end{notation}

Equipped these trees, we can let the direct sum $\bigoplus 
_{v \in X^{\ell}} \Aut(\T_X^{[\ell]})$ act on $\T_X$ via $(\alpha_v)_{v}(uw) = u \alpha_{u}(w)$ for every $u \in X^{\ell}$ and $w \in \T_X^{[\ell]}$.
By combining the latter type of tree automorphism with the finitary ones, we can decompose each tree automorphism $\alpha \in \Aut(\T_X)$ by writing
\[
\alpha = \widehat{\pi_{\ell}(\alpha)} \cdot (\alpha_v)_{v \in X^{\ell}},
\]
where $\alpha_v \in \Aut(\T_X^{[\ell]})$ is the unique automorphism that satisfies 
\[
\alpha(vw) = \alpha(v) \alpha_v(w)
\]
for every $w \in \T_X^{[\ell]}$.
In fact, for each $\ell\in \mathbb N_0$, it can be easily seen that the map
\begin{equation}\label{eq:wreath-iso}
  \Aut(\T_X) \rightarrow \widehat{\pi_{\ell}(\Aut(\T_X))} \ltimes \bigoplus \limits_{v \in X^{\ell}} \Aut(\T_X^{[\ell]}),\ 
  \alpha \mapsto \widehat{\pi_{\ell}(\alpha)} \cdot (\alpha_v)_{v \in X^{\ell}}
\end{equation}
defines an isomorphism, which is referred to as the \emph{permutational wreath decomposition of level $\ell$} of $\Aut(\T_X)$.
For each $v \in X^{\ell}$, the automorphism $\alpha_v$ in~\eqref{eq:wreath-iso} is called the \emph{section} of $\alpha$ at $v$.
Note that the section of the product $\alpha \beta$ of two automorphisms $\alpha,\beta \in \Aut(\T_X)$ at $v$ is given by the formula
\begin{equation}\label{eq:section-formula}
(\alpha \beta)_v = \alpha_{\beta(v)} \beta_v,
\end{equation}
which we will extensively apply in what follows.
The following notation is quite useful for working with sections.

\begin{notation}\label{not:v-gamma}
Let $\ell,d \in \N_0$.
Given a vertex $w \in \T_X^{[\ell]}$ of level $d$ and an element $\gamma \in \Aut(\T_X^{[\ell+d]})$, we write $w\gamma \in \Aut(\T_X^{[\ell]})$ to denote the unique element with support in $w\T_X^{[\ell+d]}$ that is given by $w\gamma(wv) = w\gamma(v)$ for every $v \in \T_X^{[\ell+d]}$.
More generally, we write $wH \defeq \Set{wh}{h \in H} \subseteq \Aut(\T_X^{[\ell]})$ for every subset $H \subseteq \Aut(\T_X^{[\ell+d]})$.
\end{notation}

\subsection{Branch groups}\label{subsec:branch-groups} Let us now fix a subgroup $G$ of $\Aut(\T_X)$.
We say that $G$ acts \emph{spherically transitively} on $\T_X$ if the action of $G$ on $X^{\ell}$ via $\pi_{\ell}$ is transitive for every $\ell \in \N_0$.
The \emph{level $\ell$ stabilizer subgroup} in $G$ is defined by
\[
\St_G(\ell) \defeq \bigcap \limits_{v \in X^{\ell}} \St_G(v),
\]
where $\St_G(v)$ denotes the stabilizer of $v$ in $G$.
The \emph{rigid stabilizer} of $v \in X^{\ell}$ in $G$ is the subgroup $\RiSt_G(v)$ of $\St_G(v)$ that consists of those elements in $G$ that fix all vertices in $\T_X$ that are not contained in $v\,\T_{X}^{[\ell]}$.
The subgroup of $\St_G(\ell)$ that is generated by the groups $\RiSt_G(v)$ with $v \in X^{\ell}$ is called the \emph{rigid level $\ell$ stabilizer subgroup} in $G$. We denote this subgroup as $\RiSt_G(\ell)$.
If $G$ acts spherically transitively on $\T_X$, and each $\RiSt_G(\ell)$ is a finite index subgroup of $G$, then $G$ is said to be a \emph{branch subgroup} of $\Aut(\T_X)$.
A group is called a \emph{branch group} if it is isomorphic to a branch subgroup of the automorphism group of some spherically homogeneous rooted tree.

\section{Constructing branch groups}\label{sec:constructing-branch-groups}

For the rest of this section we fix a finitely generated infinite residually finite group $G$.
Our aim in this section is to construct a finitely generated branch group $\Gamma$ such that $G$ embeds as a subgroup of $\Gamma$.
We then use this embedding to prove \cref{thm:main-embedding-intro}.

\subsection{The alphabets}\label{subsec:alphabets}

In order to define our branch group $\Gamma$, we must first define the tree on which it acts, which in turn is based on a sequence of alphabets.
To describe this sequence, we fix a descending chain $(N_n)_{n\in \N}$ of finite index normal subgroups of $G$ that is residual, i.e.\ satisfies $\bigcap_{n\in \N} N_n = \{1_G\}$.
For each $n \in \N$, let $Q_n=G/N_n$ and let $o_n$ be the trivial element in $Q_n$.
Let $\lambda_n \colon G \rightarrow \Sym(Q_n)$ denote the action of $G$ on $Q_n$ by left multiplication of cosets.

For each $n\in \N$, let $x_n$, $y_n$, $z_n$, $p_n$, $q_n$ be pairwise distinct elements that are also distinct from the elements of $Q_n$ and let
\begin{equation}\label{eq:defX_n}
    X_n \defeq Q_n \cup \{x_n, y_n, z_n, p_n, q_n\}.
\end{equation}
We define an action $\varphi_n$ of $G$ on $X_n$ by setting
\begin{equation}\label{def:phi_k}
\varphi_n \colon G \rightarrow \Alt(X_n),\ \ g \mapsto
  \begin{cases}
    \lambda_n(g)              &\text{if }\sgn(\lambda_n(g))=1\\
    \lambda_n(g)\cdot ( p_k\ q_k ) &\text{if }\sgn(\lambda_n(g))=-1
  \end{cases}
\end{equation}
where $\sgn(\lambda_n(g))\in \{1,-1\}$ denotes the even/odd parity of the $\lambda_n(g) \in \Sym(Q_k)$.
Notice that the automorphisms $\varphi_n(g)$ stabilize the letters $x_n,y_n,z_n\in X_n$, and that they only permute $p_k$ and $q_k$ if doing so would result in an even permutation. 
The family of actions $(\varphi_n)_{n \in \N}$ is the first of three ingredients that we require in order to define the group $\Gamma$ in Theorem~\ref{thm:main-embedding-intro}.
To obtain the second ingredient, we consider the alternating group $A \defeq \Alt(\{x,y,z,o,p,q\})$ and let it act on each of the sets $X_n$ in the obvious way, that is, we define a homomorphism $\psi_n \colon A \rightarrow \Alt(X_n)$ by setting $\psi_n(a)(*_n) = (a(*))_n$ for each $* \in \{x,y,z,o,p,q\}$, and $\psi_n(a)(\omega)=\omega$ for each $\omega \in X_n \setminus\{x_n,y_n,z_n,o_n,p_n,q_n\}$.

\medskip

For later reference, we prove the following technical lemma.

\begin{lemma}\label{lem:computable-chain-with-Y}
Suppose that $G$ is $\EFRF^{+}$ and recursively presentable.
Then the descending chain $(N_n)_{n \in \N}$, as above, can be chosen such that 
\begin{enumerate}
    \item 
    there is a computable sequence $(f_n)_{n\in \N}$ of epimorphisms $f_n\colon G\to Q_n$ such that $\ker(f_n)=N_n$ for each $n\in \N$;
    \item 
    for each $n\in \N$ and each $g\in N_n\setminus \{1\}$, the word length of $g$ with respect to the generating set $S$ is bounded from below by $n+1$; 
    and
    \item 
    the sequences $(\varphi_n)_{n\in \N}$ and $(\psi_n)_{n\in \N}$, as described above, are computable.
\end{enumerate}
Notice here that when we say that a sequence of functions $(h_n)_{n\in \N}$ is \emph{computable}, we mean that there exists a Turing machine which takes as input a number $n\in \N$, and provides as output, a description of $h_n$ as discussed in \cref{sec:output-formats}.
\end{lemma}
\begin{proof}
From \cref{lem:computable-homos} we know that there exists a computable sequence $(f_n)_{n\in\N}$ of surjective homomorphisms $f_n\colon G\to Q_n$ such that their kernels $N_n=\ker(f_n)$ form a descending residual chain of finite-index normal subgroups, that is,
\[
    N_1\geq N_2 \geq N_3 \geq \cdots 
    \qquad
    \text{with}
    \qquad
    \bigcap_{n\in\N} N_n = \{1\}.
\]
Recall that these functions $f_n$ are surjective, in particular, $f_n(G)=Q_n$.
Moreover, from the definition of the subgroups $N_n$, we see that every nontrivial element of $N_n$ has length at least $n+1$.

Since each function $f_n$ is surjective, we see that each group $Q_n$ is computable as the finite group generated by $Q_n =\left\langle\{ f_n(s) \mid s \in S \}\right\rangle$ where $S$ is a finite set.
We then see that the sequence of functions $(\lambda_n)_{n\in \N}$, as defined at the beginning of this section, is computable as each function in the sequence is completely defined by finitely many values, i.e.,
\[
    \lambda_n(s)\cdot q = f(s)\cdot q
\]
for each $s\in S$ and $q\in Q_n$.

Since the group structures of $Q_n$ are computable, we then see that each set
\[
    X_n
    =
    Q_n \cup\{ x_n, y_n, z_n, p_n, q_n \}
\]
and the elements $o_n\in Q_n$ are also computable.
It then immediately follows that a computable description of the sequence $(\varphi_n)_{n\in\N}$, as defined in \eqref{def:phi_k}, can be constructed from our computable description of $(\lambda_n)_{n\in \N}$.

Moreover, from our computable description of $(X_n)_{n\in\N}$, it immediately follows that the sequence of functions $(\psi_n)_{n\in\N}$ is computable.
In particular, the Turing machine for such a function need only transform the action of element of the alternating group $\Alt(\{x,y,z,o,p,q\})$ to an action on the subsets of the form $\{ x_n, y_n, z_n, o_n, p_n, q_n\}\subseteq X_n$ which is known from computation of $(X_n)_{n\in \N}$.
\end{proof}

\subsection{The groups \texorpdfstring{$\Gamma_{\ell}$}{\$\textbackslash Gamma\_\textbackslash ell\$}}\label{sec:constructing-the-group-gamma}
Equipped with the sequence $X = (X_n)_{n\in\N}$, we may define the trees $\T_X^{[\ell]}$ for each $\ell\in \N_0$, as in \cref{subsec:rooted-trees}.
Our aim in this section is to define and study a sequence of branch groups $\Gamma_\ell \leq \Aut(\T_X^{[\ell]})$.
In order to define these groups, we first introduce the group $H \coloneqq G \times A$, where $A = \Alt(\{x,y,z,o,p,q\})$ is the same group as in \cref{subsec:alphabets}.
For each $\ell\in\N_0$, we define an action of $H$ on $\T_X^{[\ell]}$ via the sequences $(\varphi_n)_{n\in\N}$, $(\psi_n)_{n\in\N}$ of homomorphisms $\varphi_n\colon G\to \Alt(X_n)$ and $\psi_n\colon A\to \Alt(X_n)$ from \cref{subsec:alphabets}.
More precisely, we recursively define a sequence of homomorphisms $\widetilde{\cdot}^{[\ell]}\colon H \to \Alt(\T_X^{[\ell]})$ such that for each $h=(g,\sigma)\in H$,

\[
    \widetilde{h}^{[\ell]} = (\gamma_d)_{d \in X_{\ell+1}} \in \Alt(\T_X^{[\ell+1]})\wr \Sym(X_{\ell+1}),
\]
where
\[
\gamma_d = (\widetilde{h}^{[\ell]})_d = 
\begin{cases}
        \widetilde{h}^{[\ell+1]} &
        \text{if } d = x_{\ell+1},\\[3pt]
        \widehat{\varphi_{\ell+2}(g)} &
        \text{if } d = y_{\ell+1},\\[3pt]
        \widehat{\psi_{\ell+2}(\sigma)} &
        \text{if } d = z_{\ell+1},\\[3pt]
        \id &
        \text{otherwise}.
\end{cases}
\]
Notice here that $\widetilde{h}^{[\ell]}$ stabilizes the first level of the tree $\T_X^{[\ell]}$.
We provide \cref{fig:enter-label} to assist in the understanding of this action.
In what follows it will be convenient to extend the domains of $\varphi$ and $\psi$ to $H$ by setting $\varphi(h)\coloneqq\varphi(g)$ and $\psi(h)\coloneqq\psi(a)$ for $h = (g,a)\in H$.

\newpage

\begin{figure}[!htp]
    \centering
    \includegraphics
    [width=\linewidth]{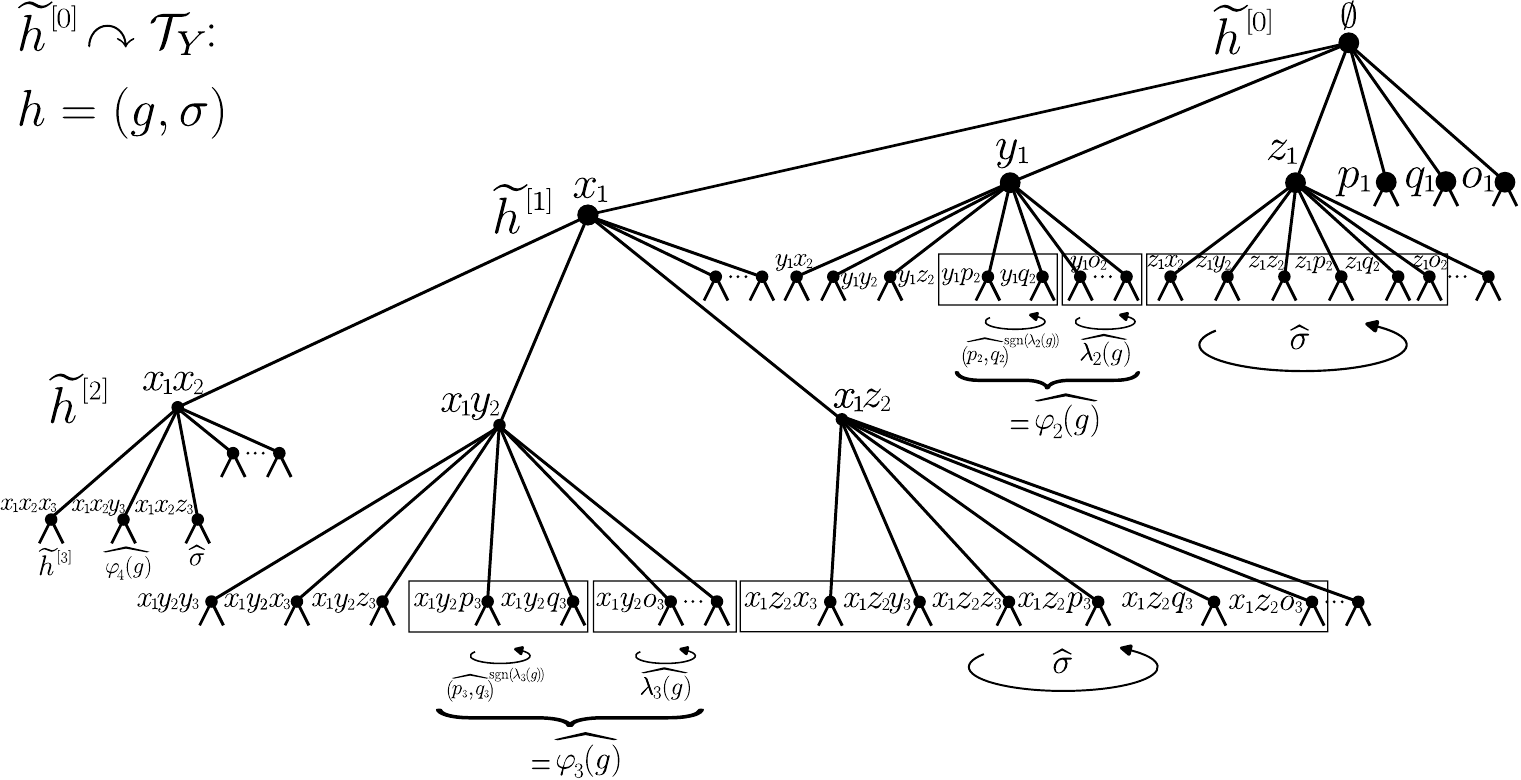}
    \caption{The action of $\widetilde{h}^{[0]}$ on $\mathcal T_X = \T_X^{[0]}$.}
    \label{fig:enter-label}
\end{figure}

\noindent In order to define our branch groups, we introduce one additional family of tree automorphisms:
For each $\ell \in \N_0$, we define
\[
B_{\ell} \defeq \widehat{\Alt(X_{\ell+1})} \leq \Aut(\T_X^{[\ell]}).
\]
That is, $B_{\ell}$ is the group of rooted automorphisms of $\T_X^{[\ell]}$ that correspond to even permutations of $X_{\ell+1}$.

For each $\ell \in \N_0$, we can now define the group
\[
    \Gamma_{\ell}
    \defeq
    \left\langle
        \widetilde{H}^{[\ell]}, B_{\ell}
    \right\rangle \leq \Aut(\T_X^{[\ell]}).
\]

That is, $\Gamma_\ell$ is the subgroup of $\Aut(\T_X^{[\ell]})$ that is generated by $\widetilde{H}^{[\ell]}$ and $B_{\ell}$.
In the remainder of this section, we show that each such group $\Gamma_\ell$ is a branch group, and we prove some lemmas which will be useful in the proofs in later sections.
We start by observing how the groups $\Gamma_{\ell}$ are related to each other.
To see this, we need the following auxiliary observation.

\begin{lemma}\label{lem:alternating}
Let $\Omega$ be a finite set and let $A,B \subseteq \Omega$ be subsets such that
\begin{enumerate}
\item $\Omega =  A \cup B$,
\item $A \cap B = \{\omega\}$ for some $\omega \in \Omega$,
\item $|A| \geq 3$.
\end{enumerate}
Let $G \leq \Alt(\Omega)$ be a subgroup that fixes at least two elements of $A \setminus \{\omega\}$ and acts transitively on $B$.
The subgroup of $\Alt(\Omega)$ that is generated by the $G$-conjugates of $\Alt(A)$ coincides with $\Alt(\Omega)$.
\end{lemma}

\begin{proof}
Let $H$ denote the subgroup of $\Alt(\Omega)$ that is generated by the $G$-conjugates of $\Alt(A)$.
Let $p,q \in A \setminus \{\omega\}$ be two distinct elements that are fixed by $G$.
If we conjugate the $3$-cycle $(p\ q\ \omega)$ with all elements in $G$, then, by the transitivity of the action of $G$ on $B$, we obtain that all $3$-cycles of the form $(p\ q\ x)$ lie in $H$, where $x \in B$.
Moreover, since $(p\ q\ x) \in \Alt(A)$ for every $x \in A$, it follows that $H$ contains the subset $\Set{(p\ q\ x)}{x \in \Omega} \subseteq \Alt(\Omega)$, which is well-known to generate $\Alt(\Omega)$.
\end{proof}

The following lemma appears in several variations within the literature (see for example~\cite{Segal01,KionkeSchesler21}) and is based on the so-called commutator trick, which was extracted by Segal \cite[Lemma 4]{Segal01} from a proof of Grigorchuk~\cite[Theorem 4]{Grigorchuk00}.

\begin{lemma}\label{lem:branch-group-arg}
For each $\ell \in \N_0$, the group $\Gamma_{\ell}$ is a finitely generated branch subgroup of $\Aut(\T_X^{[\ell]})$.
Moreover, the wreath decomposition
\[
  \Gamma_{\ell} \rightarrow \Gamma_{\ell+1} \wr_{X_{\ell+1}} B_{\ell},
  \ \ 
  \gamma \mapsto (\gamma_a)_{a \in X^{\ell+1}} \cdot \widehat{\pi_{1}(\gamma)},
\]
defines an isomorphism.
\end{lemma}

\begin{proof}
Let $\ell \in \N_0$.
Since $\Gamma_\ell$ is generated by the two finitely generated groups $\widetilde{H}^{[\ell]} = \widetilde{G\times A}^{[\ell]}$ and $B_\ell = \widehat{\Alt(X_{\ell+1})}$ it follows that $\Gamma_\ell$ itself is finitely generated.
Let us now prove that $\Gamma_{\ell}$ is a branch subgroup of $\Aut(\T_X^{[\ell]})$.
To this end, we show that for each $n \in \N$, the rigid level stabilizer subgroup $\RiSt_{\Gamma_{\ell}}(n)$ of $\Gamma_{\ell}$ coincides with the ordinary level stabilizer subgroup $\St_{\Gamma_{\ell}}(n)$, which is of finite index in $\Gamma_{\ell}$.

From their definitions, it directly follows that the $n$-th level sections of $\widetilde{H}^{[\ell]}$ and $B_{\ell}$ lie in $\Gamma_{\ell+n}$ for each $n\in \N$.
Regarding~\eqref{eq:section-formula}, this implies that all $n$th level sections of all elements in $\Gamma_{\ell}$ lie in $\Gamma_{\ell+n}$.
It is therefore enough to show that $v\Gamma_{\ell+n}$ is contained in $\RiSt_{\Gamma_{\ell}}(n)$ for every $v \in X_{\ell+1} \times \cdots \times X_{\ell+n}$.
By induction on $n$, it suffices to verify that $\RiSt_{\Gamma_{\ell}}(1) = \St_{\Gamma_{\ell}}(1)$ and $\Gamma_{\ell} \cong \Gamma_{\ell+1} \wr_{X_{\ell+1}} B_{\ell}$.

To prove the latter, let $\sigma \in B_{\ell}$ be such that $\sigma(z_{\ell+1}) = z_{\ell+1}$ and $\sigma(\{x_{\ell+1},y_{\ell+1}\}) \cap \{x_{\ell+1},y_{\ell+1}\} = \emptyset$.
Then for all $h,k \in H$ we have
\[
[(\widetilde{h}^{[\ell]})^{\sigma},\widetilde{k}^{[\ell]}]_a = 
\begin{cases}
        \widehat{\psi_{\ell+2}([h,k])} &
        \text{if } a = z_{\ell+1},\\[3pt]
        \id &
        \text{otherwise}.
\end{cases}
\]
Since $A \cong \Alt(6)$ is perfect and $\psi(H) \cong A$, we deduce that $z_{\ell+1}(\widehat{\psi_{\ell+2}(A)})$ lies in $\RiSt_{\Gamma_{\ell}}(z_{\ell+1})$.
Moreover, since $B_{\ell}$ acts transitively on $Q_{\ell+1}$, there is a $\tau \in B_{\ell}$ with
\[
(z_{\ell+1}\widehat{\psi_{\ell+2}(A)})^{\tau} = y_{\ell+1}\widehat{\psi_{\ell+2}(A)} \leq \Gamma_{\ell}.
\]
Now we can conjugate $y_{\ell+1}\widehat{\psi_{\ell+2}(A)}$ by $\widetilde{h}^{[\ell]}$ to deduce that
\[
\left(y_{\ell+1}\widehat{\psi_{\ell+2}(A)}\right)^{\widetilde{h}^{[\ell]}}
= y_{\ell+1}\widehat{\psi_{\ell+2}(A)^{\varphi_{\ell+2}(h)}}
\]
lies in $\Gamma_{\ell}$.
Using that the action of $G$ on $Q_{\ell+2}$ via $\varphi_{\ell+2}$ is transitive, we can apply Lemma~\ref{lem:alternating} to conclude that $y_{\ell+1}B_{\ell+1} \leq \Gamma_{\ell}$.
As a consequence, it follows that
\[
\widetilde{h}^{[\ell]}
\cdot \left(y_{\ell+1}\varphi_{\ell+2}(h)^{-1}\right)
\cdot \left(z_{\ell+1}\psi_{\ell+2}(h)^{-1}\right)
= x_{\ell+1} \widetilde{h}^{[\ell+1]} \in \Gamma_{\ell}
\]
for every $h \in H$.
Thus $x_{\ell+1} \widetilde{H}^{[\ell+1]}$ lies in $\Gamma_{\ell}$.
By conjugating $x_{\ell+1} \widetilde{H}^{[\ell+1]}$ with an appropriate element from $B_{\ell}$ we obtain $y_{\ell+1} \widetilde{H}^{[\ell+1]} \leq \Gamma_{\ell}$ and hence
\[
y_{\ell+1} \left\langle \widetilde{H}^{[\ell+1]},B_{\ell+1} \right\rangle
= y_{\ell+1} \Gamma_{\ell+1}
\leq \Gamma_{\ell}.
\]
Taking again conjugates by elements of $B_{\ell}$ we finally get $a \Gamma_{\ell+1} \leq \Gamma_{\ell}$ for every $a \in X_{\ell+1}$, which we wanted to prove.
This gives us
\[
\Gamma_{\ell}
= \St_{\Gamma_{\ell}}(1) \rtimes B_{\ell}
= \bigoplus \limits_{a \in X_{\ell+1}} a\Gamma_{\ell+1} \rtimes B_{\ell}
= \Gamma_{\ell+1} \wr_{X_{\ell+1}} B_{\ell},
\]
which completes the proof.
\end{proof}

In the remainder of this section, we introduce one additional tool which will be useful in \cref{sec:eff-res-fin}.
We begin by defining the group of \emph{germs} as follows.

\begin{definition}\label{def:germ}
Let $G$ be a group that acts on a topological space $\Omega$ and let $\omega$ be a point in $\Omega$.
The group of \emph{germs} of $G$ at $\omega$ is defined as $\mathcal{O}_{G}(\omega) = \St_{G}(\omega) / \sim$, where $g \sim h$ if and only if $g$ and $h$ coincide on some neighbourhood of $\omega$.
Let $\pi_{\omega} \colon \St_{G}(\omega) \rightarrow \mathcal{O}_G(\omega)$ denote the canonical quotient map.
\end{definition}

\begin{lemma}\label{lem:embedding-into-germ}
Let $\ell \in \N_0$ and let $\omega_{\ell} \defeq (x_{\ell+1},x_{\ell+2},\ldots) \in \partial \T_Y^{[\ell]}$.
The map
\[
\iota_{\ell} \colon H \rightarrow \mathcal{O}_{\Gamma_{\ell}}(\omega_{\ell}),
\ \ h \mapsto \pi_{\omega_{\ell}}(\widetilde{h}^{[\ell]})
\]
is injective.
\end{lemma}
\begin{proof}
It is a direct consequence of the definition of $\widetilde{\cdot}^{[\ell]}$ that $\widetilde{h}^{[\ell]}$ fixes $\omega_{\ell}$, which shows that $\iota_{\ell}$ is well-defined.
Suppose now that $\pi_{\omega_{\ell}}(\widetilde{h}^{[\ell]}) = \id$, i.e.\ that there is a neighbourhood $U$ of $\omega_{\ell}$ on which $\widetilde{h}^{[\ell]}$ coincides with the identity.
Recall that a neighbourhood base of $\omega_{\ell}$ in $\partial \T_X^{[\ell]}$ is given by $(x_{\ell+1} \cdots x_{n} \T_X^{[n]})_{n > \ell}$.
Thus we can choose $n$ to be such that $x_{\ell+1} \cdots x_{n} \T_X^{[n]} \subseteq U$.
Since the section of $\widetilde{h}^{[\ell]}$ at $x_{\ell+1} \cdots x_{n}$ is given by $\widetilde{h}^{[n]}$, we deduce that $\widetilde{h}^{[n]}$ is the identity on $\T_X^{[n]}$.
In that case it follows that $\varphi_{n+k}(h)$ and $\psi_{n+k}(h)$ coincide with the identity in $\Alt(X_{n+k})$ for every $k \in \N$.
In other words, this tells us that $h = (g,\sigma)$ satisfies $\sigma = \id$ in $\Sym(X_{n+k})$ and that $g \in N_{n+k}$ for every $k \in \N$.
As $(N_k)_{k \in \N}$ was chosen to be a properly decreasing residual chain of $G$, we deduce that $h = 1$, which completes the proof.
\end{proof}

\section{Preservation of the conjugacy classes}

In the remainder of this paper, it will be useful to think of $\Gamma_{\ell}$ as a quotient of the free product $F_{\ell} \defeq H \ast B_{\ell}$ via the canonical projection $\pr_{\ell} \colon F_{\ell} \rightarrow \Gamma_{\ell}$, which is given by $\pr_{\ell}(h) = \widetilde{h}^{[\ell]}$ for $h \in H$ and $\pr_{\ell}(b) = b$ for $b \in B_{\ell}$.
Using $F_{\ell}$, we define a complexity function $\absH{\cdot} \colon \Gamma_\ell\to \N_0$, variants of which have turned out to be useful in many inductive arguments in the realm of branch groups, see e.g.~\cite{AlexoudasKlopschThillaisundaram16,KionkeSchesler22}.
By definition, each element $w\in F_\ell$ can be uniquely written as
\begin{equation}\label{eq:alternating-word}
  w = b_1 h_1 b_2 h_2 \cdots b_n h_n b_{n+1}
\end{equation}
where $n \geq 0$, each $b_i\in B_\ell$ with $b_j\neq 1$ for each $2\leq j\leq n$, and each $h_i\in H\setminus \{1\}$.
For each word $w$, as in \eqref{eq:alternating-word}, we then define $\absH{w}=n$, that is, $\absH{\cdot}\colon F_\ell \to \N_0$ counts the number of $H$ factors in a normal form of an element in $F_\ell$.
Notice that this function satisfies the triangle inequality, that is, $\absH{uv}\leq \absH{u}+\absH{v}$ for each $u,v\in F_\ell$.
We extend this definition to a map $\absH{\cdot}\colon \Gamma_\ell \to \N_0$ by setting
\[
  \absH{\gamma}
  \coloneqq
  \min\{
    \absH{w}
  \mid
    w \in F_\ell\text{ with }
    \pr_\ell(w) = \gamma
  \}.
\]

\begin{definition}\label{def:fragmented-subword}
Let $A$ be an alphabet and let $v,w\in A^*$.
We say that $v$ is a \emph{fragmented subword} of $w$ if $w$ can be factored as
\[
  w = u_1 v_1 u_2 v_2 \cdots u_k v_k u_{k+1},
\]
such that $v = v_1 v_2\cdots v_k$ and $u_i,v_i \in A^*$ for each $i$.
\end{definition}

\begin{lemma}\label{lem:length-contraction-base-case}
  Let $\ell \in \N_0$, and suppose that $w, w' \in F_\ell$ have the forms
  \[
    w = b_1 h_1 b_2
    \qquad\text{and}\qquad
    w' = b_1 h_1 b_2 h_2 b_3.
  \]
  We write $\gamma = \pr_\ell(w)$ and $\gamma'=\pr_\ell(w')$ for the element of $\Gamma_\ell$ corresponding to $w$ and $w'$, respectively.
  Then, for each $d\in X_\ell$, there are elements $w_d,w_d' \in F_{\ell+1}$ with
  \begin{align*}
    w_d &\in
        B_{\ell+1}
        \cup \{ h_1 \}\text{ and}\\
    w'_d &\in
        B_{\ell+1}
        \cup \{ h_1 h_2 \}
        \cup \{ h_1 \beta \mid \beta \in B_{\ell+1} \}
        \cup \{ \beta h_2 \mid \beta \in B_{\ell+1} \}
  \end{align*}
  such that $\gamma_d = \pr_{\ell + 1}(w_d)$ and $\gamma'_d = \pr_{\ell+1}(w'_d)$.
\end{lemma}

\begin{proof}
Recall from \cref{sec:constructing-branch-groups}, that the sections of $\widetilde{h}^{[\ell]}\in H$ are defined as
\[
  (\widetilde{h}^{[\ell]})_d = 
  \begin{cases}
    \widetilde{h}^{[\ell+1]} & \text{if } d = x_{\ell+1},\\[3pt]
    \widehat{\varphi_{\ell+2}(h)} & \text{if } d = y_{\ell+1},\\[3pt]
    \widehat{\psi_{\ell+2}(h)} & \text{if } d = z_{\ell+1},\\[3pt]
    \id & \text{otherwise}
  \end{cases}
\]
for each $d \in X_{\ell+1}$, where $\widehat{\varphi_{\ell+2}(h)},\widehat{\psi_{\ell+2}(h)} \in B_{\ell+1}$.
Regarding this, the lemma is a direct consequence of~\eqref{eq:section-formula}.
\end{proof}

\begin{lemma}\label{lem:length-contraction}
Let $\ell \in \N_0$, let $w \in F_{\ell}$ be a word of the form
\[
w = b_1 h_1 b_2 h_2 \cdots b_n h_n b_{n+1},
\]
where $b_i \in B_\ell$ and $h_i\in H$, and let $\gamma = \pr_\ell(w)$.
Then for each $d \in X_{\ell+1}$, we may find a word of the form
\[
w_d = \beta_1 h'_1 \beta_2 h'_2 \cdots \beta_m h'_m \beta_{m+1} \in F_{\ell+1}
\]
such that $\gamma_d = \pr_{\ell+1}(w_d)$ with $m \leq \lceil n/2 \rceil$, $\beta_i \in B_{\ell+1}$, and $h_i'\in H^*$ such that $h_1' h_2' \cdots h_m'\in H^*$ is a fragmented subword of $h_1 h_2 \cdots h_n \in H^*$.
In particular we have $\bigAbsH{\gamma_d} \leq \lceil \absH{\gamma}/2 \rceil$  for each $d\in X_{\ell+1}$.
\end{lemma}

\begin{proof}
We begin by collecting adjacent terms of $w$; in particular, depending on the even/odd parity of $n$, we can group the terms of $w$ as either
\begin{align*}
  w &= (b_1 h_1 b_2 h_2) (b_3 h_3 b_4 h_4)\cdots (b_{n-3} h_{n-3} b_{n-2} h_{n-2}) (b_{n-1} h_{n-1} b_n h_n b_{n+1})\quad\text{or}\\
  w &= (b_1 h_1 b_2 h_2) (b_3 h_3 b_4 h_4)\cdots (b_{n-2} h_{n-2} b_{n-1} h_{n-1}) (b_n h_n b_{n+1}).
\end{align*}
From \cref{lem:length-contraction-base-case} and \eqref{eq:section-formula}, we see that, for each $d\in X_{\ell+1}$, we may choose a word $w_d \in F_{\ell+1}$ with $\gamma_d = \pr_{\ell+1}(w_d)$ such that either
\[
  w_d = s_1 s_2\cdots s_{n/2}
  \qquad
  \text{or}
  \qquad
  w_d = s_1 s_2\cdots s_{\lfloor n/2\rfloor} s',
\]
where
\[
  s_k
  \in
  B_{\ell+1}
  \cup \{ h_{2k-1} h_{2k} \}
  \cup \{ h_{2k-1} \beta_2 \mid \beta_2 \in B_{\ell+1} \}
  \cup \{ \beta_1 h_{2k} \mid \beta_1 \in B_{\ell+1} \}
\]
and
$
  s'
  \in
  B_{\ell+1}
  \cup \{ h_{n} \}
$.
Notice here that each $s_i$ corresponds to a section of one of the groupings of terms, and $s'$ corresponds to a section of $b_n h_n b_{n+1}$.
It is then clear that such a word has each of the desired properties of the lemma.
\end{proof}

\begin{theorem}\label{thm:fratini-embedding}
Let $\ell \in \N_0$.
Two elements $g,k \in G$ are conjugate in $G$ if and only if the corresponding elements $\widetilde{g}^{[\ell]},\widetilde{k}^{[\ell]}\in \Gamma_\ell$ are conjugate in $\Gamma_\ell$.
  That is, for every $\ell\in \N_0$, the map $g\mapsto \widetilde{g}^{[\ell]}$ is a Frattini embedding of $G$ into $\Gamma_\ell$.
\end{theorem}

\begin{proof}
Since the map $g\mapsto \widetilde{g}^{[\ell]}$ is a homomorphism it is clear that $\widetilde{g}^{[\ell]},\widetilde{k}^{[\ell]}\in \Gamma_\ell$ are conjugate in $\Gamma_\ell$ if $g,k \in G$ are conjugate in $G$.
Suppose now that $g,k\in G$ are elements that are not conjugate in $G$.
Thus, to complete our proof, it remains to be shown that $\widetilde{g}^{[\ell]}$ and $\widetilde{k}^{[\ell]}$ are not conjugate in $\Gamma_\ell$.
Since $g\mapsto \widetilde{g}^{[\ell]}$ is injective, this property immediately follows if one of $g$ or $k$ is the identity.
In the remainder of this proof, we assume without loss of generality that neither $g$ nor $k$ is the identity of $G$.			 
Suppose for contradiction, that $\widetilde{g}^{[\ell]}$ and $\widetilde{k}^{[\ell]}$ are conjugate, that is, that there exists some $\gamma\in \Gamma_\ell$ for which $\gamma^{-1} \widetilde{g}^{[\ell]} \gamma = \widetilde{k}^{[\ell]}$.
We will show, by induction on $\absH{\gamma}$, that no such element $\gamma\in \Gamma_\ell$ can exist.
We begin with the base cases $\absH{\gamma} = 0$ and $\absH{\gamma}=1$, as follows.

\medskip

\noindent\underline{Base case}: $\absH{\gamma} = 0$.
\nopagebreak

\nopagebreak\smallskip
\nopagebreak\noindent
To begin, suppose that $\absH{\gamma}=0$, that is, $\gamma\in B_\ell$.

Recall from the definition of $\widetilde{\cdot}^{[\ell]}$ that
$(\widetilde{g}^{[\ell]})_{x_{\ell+1}} = \widetilde{g}^{[\ell+1]}$ and
$(\widetilde{k}^{[\ell]})_{x_{\ell+1}} = \widetilde{k}^{[\ell+1]}$, neither of which are finitary;
and that for each $a\in X_{\ell+1}\setminus\{x_{\ell+1}\}$, we have $(\widetilde{g}^{[\ell]})_{a} \in B_{\ell+1}$ and
$(\widetilde{k}^{[\ell]})_{a} \in B_{\ell+1}$, which are finitary.
Moreover, notice that, if $\gamma^{-1} \widetilde{g}^{[\ell]} \gamma = \widetilde{k}^{[\ell]}$, then
\[
(\gamma^{-1} \widetilde{g}^{[\ell]} \gamma)_{x_{\ell+1}}
= \widetilde{k}^{[\ell+1]}.
\]
Hence, it must be the case that $\gamma\cdot x_{\ell+1}=x_{\ell+1}$ and thus $\widetilde{g}^{[\ell+1]} = \widetilde{k}^{[\ell+1]}$.
Since the map $\widetilde{\cdot}^{[\ell+1]} \colon G \to \Aut(\T_X^{[\ell+1]})$ is injective, this would imply that $g=k$ which contradicts our assumption that $g$ and $k$ are not conjugate in $G$.

\medskip

\noindent\underline{Base case}: $\absH{\gamma} = 1$.
\nopagebreak

\nopagebreak\smallskip
\nopagebreak\noindent
Suppose that $\absH{\gamma}=1$, then there exists an element $w\in F_\ell$ with $\gamma = \pr_\ell(w)$ and
\begin{equation}\label{eq:base-case1}
    w = b_1 h_1 b_2\in F_\ell,
\end{equation}
where $b_i\in B_\ell$ and $h_1 \in H$.
Since $\widetilde{g}^{[\ell]}\in \St_{\Gamma_\ell}(1)$ and $\gamma^{-1} \widetilde{g}^{[\ell]} \gamma = \widetilde{k}^{[\ell]}$, there exist $a,b\in X_{\ell+1}$ such that
\begin{equation}\label{eq:eq1}
    \bigl(\gamma_a\bigr)^{-1}
    \cdot (\widetilde{g}^{[\ell]})_b
    \cdot \gamma_a
    =
    (\widetilde{k}^{[\ell]})_{x_{\ell+1}} = \widetilde{k}^{[\ell+1]}.
\end{equation}
From the definition of $\widetilde{g}^{[\ell]}$, we then see that either $(\widetilde{g}^{[\ell]})_b \in B_{\ell+1}$ or $(\widetilde{g}^{[\ell]})_b = \widetilde{g}^{[\ell+1]}$.
Further, from \cref{lem:length-contraction-base-case}, we see that \[\gamma_{a} \in B_{\ell+1} \cup \left\{\widetilde{h_1}^{[\ell+1]}\right\}.\]
From (\ref{eq:eq1}), the remainder of the proof of this base case falls into 4 cases as follows.

\medskip

\noindent\textit{Case 1:}
$\tau^{-1} \cdot \widetilde{g}^{[\ell+1]}\cdot \tau = \widetilde{k}^{[\ell+1]}$ for some $\tau\in B_{\ell+1}$.

\smallskip\noindent
In this case, our contradiction follows from the previous base case.

\medskip

\noindent\textit{Case 2:}
$\tau^{-1}\cdot \sigma\cdot \tau = \widetilde{k}^{[\ell+1]}$ for some $\sigma,\tau\in B_{\ell+1}$.

\smallskip\noindent
In this case, the contradiction follows as $\widetilde{k}^{[\ell+1]}$ is not finitary and in particular does not belong to $B_{\ell+1}$ for $k\neq 1$.

\medskip

\noindent\textit{Case 3:}
$(\widetilde{h_1}^{[\ell+1]})^{-1}\cdot (\widetilde{g}^{[\ell+1]})\cdot (\widetilde{h_1}^{[\ell+1]}) = \widetilde{k}^{[\ell+1]}$.

\smallskip\noindent
Since $\widetilde{\cdot}^{[\ell+1]} \colon H\to\Gamma_{\ell+1}$ is injective, it would then follow that $h_1^{-1} g h_1 = k$ in $H=G\times A$.
Since $H$ is a direct product of $G$ and $A$, with $h_1 = (x,y)\in G\times A = H$, we see that this implies that $x^{-1} g x = k$ in $G$.
This contradicts our assumption that $g$ and $k$ are not conjugate in $G$.

\medskip

\noindent\textit{Case 4:}
$(\widetilde{h_1}^{[\ell+1]})^{-1}\cdot \sigma\cdot (\widetilde{h_1}^{[\ell+1]}) = \widetilde{k}^{[\ell+1]}$ for some $\sigma\in B_{\ell+1}$.

\smallskip\noindent
If $\sigma = 1$, then $\widetilde{k}^{[\ell+1]} = 1$, which would contradict our assumption that $k$ is non-trivial.
Suppose therefore that $\sigma \in B_{\ell+1}\setminus \{1\}$.
Then, a contradiction arises as
\[									 
(\widetilde{h_1}^{[\ell+1]})^{-1}\cdot \sigma\cdot (\widetilde{h_1}^{[\ell+1]})\notin \St_{\Gamma_{\ell+1}}(1)
\qquad
\text{while}
\qquad
\widetilde{k}^{[\ell+1]}\in \St_{\Gamma_{\ell+1}}(1).
\]
This completes all subcases of the base case $\absH{\gamma}=1$.

\medskip

\noindent\underline{Inductive step}:
\nopagebreak

\nopagebreak\smallskip
\nopagebreak\noindent
We now assume for induction that there is some $m\in \N$ such that for every $\ell\in\N_0$ there is no word $\gamma'\in\Gamma_\ell$ with $\absH{\gamma'}\leq m$ such that $(\gamma')^{-1} \widetilde{g}^{[\ell]} \gamma' = \widetilde{k}^{[\ell]}$.
Suppose for contradiction that there exists $\gamma\in \Gamma_\ell$ with $\absH{\gamma}=m+1$ for which $\gamma^{-1} \widetilde{g}^{[\ell]} \gamma = \widetilde{k}^{[\ell]}$.
By the definition of $\absH{\cdot}$, there must exist some element $w\in F_\ell$ which can be written as
\[
    w = b_1 h_1 b_2 h_2 \cdots b_m h_m b_{m+1} h_{m+1} b_{m+2} 
\]
where $b_i\in B_\ell$, $h_i\in H$, and $\gamma = \pr_\ell(w)$.

Notice that if $\gamma^{-1} \widetilde{g}^{[\ell]} \gamma = \widetilde{k}^{[\ell]}$, then it would follow that

\begin{equation}\label{eq:projections}
    \bigl(\gamma^{-1} \widetilde{g}^{[\ell]} \gamma\bigr)_{x_{\ell+1}}
    =
    (\widetilde{k}^{[\ell]})_{x_{\ell+1}}
    =
    \widetilde{k}^{[\ell+1]}.
\end{equation}
Let $\beta\in B_\ell$ be the unique element for which
$
    \gamma\beta\in \St_{\Gamma_\ell}(1).
$
From \eqref{eq:section-formula}, \eqref{eq:projections} and the definition of $\widetilde{\cdot}^{[\ell]}\colon H \to \Gamma_\ell$, we see that there exists some $d\in X_{\ell+1}$ for which
\begin{equation}\label{eq:projections2}
    \bigl((\gamma\beta)_{d}\bigr)^{-1}
    \cdot (\widetilde{g}^{[\ell]})_d
    \cdot (\gamma\beta)_{d}
    =
    \widetilde{k}^{[\ell+1]}.
\end{equation}
Note that we either have $(\widetilde{g}^{[\ell]})_d\in B_\ell$ or $(\widetilde{g}^{[\ell]})_d = \widetilde{g}^{[\ell+1]}$.
We consider these two possibilities in separate cases.

\medskip

\noindent\textit{Case 1}: $(\widetilde{g}^{[\ell]})_d\in B_\ell$.
\nopagebreak

\nopagebreak\smallskip
\nopagebreak\noindent
Same arguments as in case 4 of the proof of the base case $\absH{\gamma}=1$.

\medskip

\noindent\textit{Case 2}: $(\widetilde{g}^{[\ell]})_d = \widetilde{g}^{[\ell+1]}$.
\nopagebreak

\nopagebreak\smallskip
\nopagebreak\noindent
From \cref{lem:length-contraction}, we see that $\bigAbsH{(\gamma\beta)_d}< \absH{\gamma}$.
Thus, we have a $\gamma'\in \Gamma_{\ell+1}$ with $\absH{\gamma'}< \absH{\gamma}$ for which $\gamma'^{-1} \widetilde{g}^{[\ell+1]}\gamma' = \widetilde{k}^{[\ell+1]}$ which contradicts our induction hypothesis.

\medskip

\noindent\underline{Conclusion}:
\nopagebreak

\nopagebreak\smallskip
\nopagebreak\noindent
In all cases, we see that a contradiction arises if we assume that there exists some $\gamma\in \Gamma_{\ell}$ for which $\gamma^{-1} \widetilde{g}^{[\ell]}\gamma = \widetilde{k}^{[\ell]}$ where $g$ and $k$ are not conjugate in $G$.
\end{proof}

We obtain the following result as an immediate corollary to \cref{thm:fratini-embedding}.

\begin{corollary}\label{cor:conjugate-iff}
Let $\ell \in \N_0$.
If $G$ has an unsolvable conjugacy problem, then $\Gamma_\ell$ has an unsolvable conjugacy problem.
\end{corollary}

\section{Preservation of effective residual finiteness}\label{sec:eff-res-fin}

Our aim in this section is to prove the second part of Theorem~\ref{thm:main-embedding-intro}.
In our current setting this ends up in showing that $\EFRF^{+}$ and the recursive presentability of $G$ can be extended to $\Gamma_{0}$.
Let us start by recalling and extending the notation from the previous sections.

\subsection{Setting the stage}

Let $G$ be a finitely generated infinite residually finite group and let $S$ be a finite symmetric generating set of $G$.
We assume from now on that $G$ is recursively presented and $\EFRF^{+}$.
In this case Lemma~\ref{lem:computable-chain-with-Y} provides us with a sequence $(N_n)_{n \in \N}$ of finite index normal subgroups of $G$ such that the following hold.
\begin{enumerate}
\item There exists a computable sequence $(f_n)_{n \in \N}$ of surjective homomorphisms $f_n \colon G \rightarrow Q_n$ with $\ker(f_n) = N_n$ for every $n \in \N$.
\item For each $n \in \N$ and each $g \in N_n \setminus \{1\}$ the word length of $g$ with respect to $S$ is bounded from below by $n+1$.\label{eq:word-length-bounded-below}
\item The sequences $X = (X_n)_{n \in \N}$ and $Q = (Q_n)_{n \in \N}$, where
\[
X_n \defeq Q_n \cup \{x_n,y_n,z_n,p_n,q_n\}
\]
are computable.\label{eq:computable-alphabets}
\item The sequences $(\varphi_n)_{n \in \N}$, $(\psi_n)_{n \in \N}$
of the homomorphisms
\[
\varphi_n,\psi_n \colon G \rightarrow \Alt(X_n)
\]
from Section~\ref{sec:constructing-branch-groups} are computable.\label{eq:computable-phi-and-psi}
\end{enumerate}

\medskip

\noindent Recall that for each $\ell \in \N_0$, the group $\Gamma_{\ell}$ is defined as the subgroup of $\Aut(\T_X^{[\ell]})$ that is generated by $\widetilde{H}^{[\ell]}$ and $B_{\ell}$.
Thus a finite symmetric generating set of $\Gamma_{\ell}$ is given by $\Sigma_{\ell} \defeq ((\widetilde{S}^{[\ell]} \times \widetilde{A}^{[\ell]}) \cup B_{\ell}) \setminus \{1\}$.

\begin{lemma}\label{lem:effectiveness-via-stabilizer-quotients}
Let $w \in \Sigma_{0}^{\ast}$ be a word of length $\ell \in \N_0$.
Then $w$ represents the trivial element in $\Gamma_0$ if and only if $w$ represents an element in $\St_{\Gamma_{0}}(2\ell)$.
\end{lemma}
\begin{proof}
Since the claim is trivial otherwise, we may assume that $\ell > 0$.
Let $\gamma \in \Gamma_0$ be the element represented by $w$.
Suppose that $\gamma \in \St_{\Gamma_{0}}(2\ell)$.
We want to show that $\gamma$ is the trivial element.
To this end, we write
\[
w = v_1 w_{1} \cdot \ldots \cdot v_m w_{m} v_{m+1},
\]
where $v_i \in B_0^{\ast}$ and
$w_i \in (S \times A)^{\ast}$ for every $i$.
Without loss of generality we may assume that the words $w_{i}$ are non-empty.
Then
\[
m \leq \sum \limits_{i=1}^m \abs{w_i} \leq \abs{w} = \ell.
\]
From Lemma~\ref{lem:length-contraction} we know that for every $a \in X_1$, the section $\gamma_a$ is represented by a word of the form
\[
w_a = v_1' w_{1}' \cdot \ldots \cdot v_{m'}' w_{m'}' v_{m'+1}',
\]
where $v_i' \in B_1^{\ast}$, $w_{1}' \ldots w_{m'}'$ is a fragmented subword of $w_{1} \ldots w_{m}$, and $m' < m$ if $m \geq 2$.
Using the fact that the section of $\gamma$ at a word of the form $u_1u_2$ is given by $\gamma_{u_1u_2} = (\gamma_{u_1})_{u_2}$, an inductive application of Lemma~\ref{lem:length-contraction} tells us that for every $u \in X_1 \times \ldots \times X_{\ell}$, the section of $\gamma$ at $u$ is represented by a word of the form
\[
w_u = v_1^{(u)} w^{(u)} v_2^{(u)},
\]
where $v_1^{(u)},v_2^{(u)} \in B_{\ell}^{\ast}$ and $w^{(u)}$ is a fragmented subword of $w_{1} \ldots w_{m}$.
Note that the latter implies that the word length of $w^{(u)} \in (S \times A)^{\ast}$ is bounded above by $m \leq \ell$.
As a consequence, we have
\[
\gamma_{u}
= \tau_1 \cdot \widetilde{(s_1,\sigma_1)}^{[\ell]} \cdot \ldots \cdot \widetilde{(s_k,\sigma_k)}^{[\ell]} \cdot \tau_2
\]
for some $k \leq \ell$ and appropriate elements $s_i \in S$, $\sigma_i \in A$, and $\tau_1,\tau_2 \in B_{\ell}$.
Let $g \in G$ be the element represented by $s_1 \cdot \ldots \cdot s_k$ and let $\sigma = \sigma_1 \cdot \ldots \cdot \sigma_k$.
Then
\[
\gamma_{u}
= \tau_1 \cdot \widetilde{(g,\sigma)}^{[\ell]} \cdot \tau_2
= \tau_1 \cdot \widetilde{(g,\sigma)}^{[\ell]} \cdot \tau_1^{-1} \cdot \tau_1 \tau_2.
\]
From our assumption that $\gamma \in \St_{\Gamma_{0}}(2\ell)$ it follows that $\tau_1 \tau_2 = \id$ and hence that $\widetilde{(g,\sigma)}^{[\ell]}$ lies in $\St_{\Gamma_{\ell}}(\ell)$.
Regarding the definition of the map $\widetilde{\cdot}^{[\ell]}$, we see that
\[
\phi_{k}(g) = \id \in \Sym(X_{k}) \text{ and } \psi_{k}(\sigma) = \id \in \Sym(X_{k})
\]
for every $\ell < k \leq 2\ell$.
Since $\psi_k \colon A \rightarrow \Sym(X_{k})$ is injective, the latter shows that $\sigma = \id$.
To see that $g$ is the trivial element, we recall from~\eqref{eq:word-length-bounded-below} that the word length of $g$ with respect to $S$ would be at least $\ell+1$ if $g$ is non-trivial.
However, the word length of $g$ is bounded from above by $k \leq m \leq \ell$, which implies that $g = 1$ and hence $\gamma_u = 1$.
Since $u \in X_1 \times \ldots \times X_{\ell}$ was arbitrary, it follows that the section of $\gamma$ at each word of length $\ell$ is trivial.
Together with our assumption that $\gamma \in \St_{\Gamma_{0}}(2\ell)$ this implies $\gamma = 1$, which completes the proof.
\end{proof}

With the above lemma, we can now prove the following result.

\begin{proposition}\label{prop:Gamma-eff-res-fin}
The group $\Gamma_0$ is $\EFRF^{+}$ and recursively presentable.
\end{proposition}
\begin{proof}
From Lemma~\ref{lem:effectiveness-via-stabilizer-quotients} we know that a word $w$ of length $\ell$ over $\Sigma_{0}$ represents a non-trivial element in $\Gamma_0$ if and only if $\pi_{2 \ell}(w) \neq \id$, where
\[
\pi_{2 \ell} \colon \Gamma_0 \rightarrow \Sym(X_1 \times \ldots \times X_{2\ell})
\]
denotes the action of $\Gamma_0$ on the $2\ell$-level of the tree $\T_X^{[0]}$.
Regarding this, it suffices to show that there exists a Turing machine $M$ that takes as input a word $w \in \Sigma_{0}^{\ast}$ and returns the homomorphism $\pi_{2 \ell}$, where $\ell$ is the word length of $w$.
									   
Recall that to return $\pi_{2 \ell}$ it suffices to encode the multiplication table of $\Sym(X_1 \times \ldots \times X_{2\ell})$ along with the elements $\pi_{2 \ell}(\gamma)$ for each $\gamma \in \Sigma_{0}$ as the output of a Turing machine.
To see that there is a Turing machine doing the latter, recall from~\eqref{eq:computable-alphabets} that the sequence $X = (X_{\ell})_{\ell \in \N}$ is computable.
Hence the sequence $(X^{2\ell})_{\ell \in \N}$, where $X^{2\ell} = X_1 \times \ldots \times X_{2\ell}$, is computable as well, which in turn implies that the sequence of groups $(\Sym(X^{2\ell}))_{\ell \in \N}$ is computable.
Thus it remains to show that for each generator
\[
\gamma \in \Sigma_0 = ((\widetilde{S}^{[0]} \times \widetilde{A}^{[0]}) \cup B_{0}) \setminus \{1\},
\]
the assignment $\ell \mapsto \pi_{2\ell}(\gamma)$ is computable.
To achieve this, it suffices to show that there is a Turing machine that takes as input a number $\ell \in \N$ and a word
\[
v = a_{i_1} \ldots a_{i_{2\ell}} \in X_1 \times \ldots \times X_{2\ell}
\]
and returns the word $\gamma(v)$.
In the case where $\gamma = \widehat{\tau} \in B_{0} = \widehat{\Alt(X_{1})}$
we have
\[
\widehat{\tau}(v) = \tau(a_{i_1}) a_{i_2} \ldots a_{i_{2\ell}}.
\]
To describe $\gamma(v)$ in the case where $\gamma = (\widetilde{s}^{[0]},\widetilde{\sigma}^{[0]}) \in \widetilde{S}^{[0]} \times \widetilde{A}^{[0]}$, we consider the maximal $k \in \N_0$ with the property that $a_{i_k} = x_k$.
In that case there is a word $u \in X_{k+3} \times \ldots \times X_{2\ell}$ with
\[
v = x_1 \ldots x_{k} a_{i_{k+1}} a_{i_{k+2}} u
\]
and we have
\[
(\widetilde{s}^{[0]},\widetilde{\sigma}^{[0]})(v) =
\begin{cases}
a_{i_{1}} \ldots a_{i_{k+1}} \varphi_{k+2}(s)(a_{i_{k+2}}) u &\text{if } k \leq 2\ell-2 \text{ and } a_{i_{k+1}} = y_{k+1}\\
a_{i_{1}} \ldots a_{i_{k+1}} \psi_{k+2}(\sigma)(a_{i_{k+2}}) u &\text{if } k \leq 2\ell-2 \text{ and } a_{i_{k+1}} = z_{k+1} \\
v &\text{otherwise.}
\end{cases}
\]
Since the sequences $(\varphi_k)_{k \in \N}$ and $(\psi_k)_{k \in \N}$ are computable by~\eqref{eq:computable-phi-and-psi}, we deduce from the above descriptions of $\gamma(v)$ that the assignment $\ell \mapsto \pi_{2\ell}(\gamma)$ is computable for each $\gamma \in \Sigma_0$.
To complete the proof, it remains to note that the Turing machine that takes a word $w \in \Sigma_{0}^{\ast}$ as input and returns the homomorphism $\pi_{2 \ell}$ can check whether $w$ represents the trivial word in $\Gamma_0$ by evaluating $\pi_{2 \ell}(w)$.
Since this solves the word problem in $\Gamma_0$, it follows that $\Gamma_0$ is recursively presented, which completes the proof.
\end{proof}

We have now all ingredients to deduce Theorem~\ref{thm:main-embedding-intro} from the introduction.

\EmbeddingTheorem*

\begin{proof}
From Theorem~\ref{thm:fratini-embedding}, we know that the map
\[
\iota \colon G \rightarrow \Gamma_0,\ g \mapsto \widetilde{g}^{[0]}
\]
defines a Frattini embedding.
In the case where $G$ is recursively presented and $\EFRF^{+}$ we know from Proposition~\ref{prop:Gamma-eff-res-fin} that $\Gamma_0$ is also recursively presented and $\EFRF^{+}$.
\end{proof}

We can now prove \cref{thm:main-existence-branch} from \cref{thm:main-embedding-intro} as follows.

\MainTheorem*

\begin{proof}
By a result of Miller~\cite[Theorem~9 on p.~31]{Miller1971} there exists a finitely presented residually finite group $G$ with solvable word problem and unsolvable conjugacy problem.
Since finitely presented residually finite groups are $\EFRF^{+}$, we can apply Theorem~\ref{thm:main-embedding-intro} to deduce that there exists a finitely generated recursively presented branch group $\Gamma$ with $\EFRF^{+}$ and a Frattini embedding $\iota \colon G \rightarrow \Gamma$.
As a consequence, the unsolvability of the conjugacy problem of $G$ passes to $\Gamma$ via the Frattini embedding, while $\Gamma$ has a solvable word problem since it is recursively presented and $\EFRF^{+}$.
\end{proof}

\bibliographystyle{amsplain}
\bibliography{literatur}

\end{document}